\documentclass{amsproc}

\newtheorem{theorem}{Theorem}[section]

\newtheorem{corollary}[theorem]{Corollary}

\theoremstyle{definition}
\newtheorem{definition}{Definition}

\theoremstyle{remark}
\newtheorem{remark}[theorem]{Remark}
\newtheorem{conjecture}[theorem]{Conjecture}

\numberwithin{equation}{section}

\def\phi{\varphi}

\def\A{\mathbb A}
\def\N{\mathbb N}
\def\kx{k[X_0,\ldots,X_n]}
\def\fp{{\mathcal P}}
\def\mm{{\bf m}}
\newcommand{\rar}{\rightarrow}
\newcommand{\lar}{\longrightarrow}

\begin{document}

\title{Minimal free resolutions for certain affine monomial curves}

\author{Philippe Gimenez}
\address{Department of Algebra, Geometry and
Topology, Faculty of Sciences, University of Valladolid, 47005
Valladolid, Spain.} \email{pgimenez@agt.uva.es}
\thanks{The first author was partially supported by
MTM2007-61444, {\it Ministerio de Educaci\'on y Ciencia}, Spain.}

\author{Indranath Sengupta}
\address{School of Mathematical Sciences, RKM
Vivekananda University, Belur, India.}
\curraddr{Department of Mathematics, Jadavpur University, Kolkata,
WB 700 032, India.} \email{sengupta.indranath@gmail.com}
\thanks{The second author thanks DST, Government of India for financial support for the
project ``Computational Commutative Algebra", reference no.
SR/S4/MS: 614/09, and for the BOYSCAST 2003 Fellowship. This
collaboration is the outcome of the second author's visits to the
University of Missouri, Columbia, USA during Fall 2007 and to the
University of Valladolid, SPAIN for two months in 2009 with the
research fellowship {\it Ayuda para la Estancia de Investigadores de
Otras Instituciones}. He thanks both the institutions for their
support.}

\author{Hema Srinivasan}
\address{Mathematics Department, University of
Missouri, Columbia, MO 65211, USA.}
\email{SrinivasanH@missouri.edu}

\subjclass{Primary 13D02; Secondary 13A02, 13C40.}
\date{}

\dedicatory{This paper is dedicated to Wolmer V. Vasconcelos.}

\keywords{Monomial curves, arithmetic sequences, determinantal
ideals, Betti numbers, minimal free resolutions.}

\begin{abstract}
Given an arbitrary field $k$ and an arithmetic sequence of positive
integers $m_0<\ldots<m_n$, we consider the affine monomial curve in
$\A_k^{n+1}$  parameterized by $X_0=t^{m_0}$, \ldots, $X_n=t^{m_n}$.
In this paper, we conjecture that the Betti numbers of its
coordinate ring are completely determined by $n$ and the value of
$m_0$ modulo $n$. We first show that the defining ideal of the
monomial curve can be written as a sum of two determinantal ideals.
Using this fact, we describe the minimal free resolution of the
coordinate ring in the following three cases: when $m_0\equiv 1\mod
n$ (determinantal), when $m_0\equiv n\mod n$ (almost determinantal),
and when $m_0\equiv 2\mod n$ and $n=4$ (Gorenstein of codimension
$4$).
\end{abstract}

\maketitle

\section*{Introduction}
Let $k$ denote an arbitrary field and $R$ be the polynomial ring
$\kx$. Consider the $k$-algebra homomorphism $\phi:R\rar k[t]$ given
by $\phi(X_i)=t^{m_i}$, $i=0,\ldots,n$. Then the ideal
$\fp:=\ker{\phi}\subset R$ is the defining ideal of the monomial
curve in $\A_k^{n+1}$ given by the parametrization $X_0=t^{m_0}$,
\ldots, $X_n=t^{m_n}$. The $k$-algebra of the semigroup
$\Gamma\subset\N$ generated by $m_0,\ldots,m_n$ is
$k[\Gamma]:=k[t^{m_0},\ldots,t^{m_n}]\simeq R/\fp$, which is
one-dimensional and  $\fp$ is a perfect ideal of codimension $n$. It
is well known that $\fp$ is minimally generated by binomials.
Moreover, $\fp$ is a homogeneous ideal and $k[\Gamma]$ is the
homogeneous coordinate ring if we give weight $m_i$ to the variables
$X_i$. Henceforth, homogeneous and graded would mean homogeneous and
graded with respect to this weighted graduation.

\medskip

Assume now that the positive integers $m_0,\ldots,m_n$ satisfy the
following properties:
\begin{enumerate}
\item[(i)]
$\gcd{(m_0,\ldots,m_n)}=1$;
\item[(ii)]
$0<m_0<\cdots<m_n$ and $m_i=m_0+id$ for every $i\in[1,n]$, i.e., the
integers form an arithmetic progression with common difference $d$;
\item[(iii)]
$m_0,\ldots,m_n$ generate the semigroup
$\Gamma:=\displaystyle{\sum_{0\leq i\leq n}\N m_i}$ minimally, where
$\N=\{0,1,2,\ldots\}$, i.e., $m_j\notin\displaystyle{\sum_{0\leq
i\leq n;\,i\neq j}\N m_i}$ for every $i\in[0,n]$.
\end{enumerate}

\begin{definition}\label{arithmeticsequence}\rm
A sequence of positive intergers
$(\mm ) = m_0,\ldots,m_n$ is called an {\bf arithmetic sequence} if it
satisfies the conditions (i),(ii) and (iii) above.
\end{definition}

Let us write $m_0=an+b$ such that $a,b$ are positive integers and
$b\in[1,n]$. Note that $b\in [1,n]$ and condition (iii) on
$m_0,\ldots,m_n$ ensures that $a\geq 1$; otherwise $m_0=b$ and
$m_b=m_0+bd=(1+d)b$ contradicts minimally condition (iii).

\medskip

Let $(\mm )= m_0,\ldots, m_n$ be an arithmetic sequence. We say that
the monomial curve  in $\A_k^{n+1}$  parameterized by $X_0=t^{m_0}$,
\ldots, $X_n=t^{m_n}$ is the monomial curve associated to the
arithmetic sequence $(\mm )$ and denote it by $C(\mm )$. A minimal
binomial generating set for the defining ideal $\fp$ of $C(\mm )$
was given in \cite{patil}, and it was rewritten in \cite{malooseng}
to prove that $\fp$ is not a complete intersection if $n\geq 3$. An
explicit formula for the type of $k[\Gamma]$ is given in
\cite[Corollary~6.2]{patseng} under a more general assumption of
almost arithmetic sequence on the integers $(\mm )= m_0,\ldots, m_n$
and it follows from this result that if $(\mm )= m_0,\ldots, m_n$ is
an arithmetic sequence then $k[\Gamma]$ is Gorenstein if and only if
$b=2$. In this paper, we prove in Theorem~{\ref{mingen}} that ideal
$\fp$ has the special structure that it can be written as a sum of
two determinantal ideals, one of them being the defining ideal of
the rational normal curve. We exploit this structure to construct an
explicit minimal free resolution of the graded ideal $\fp$, in the
cases when $m_0\equiv 1$ or $n$ modulo $n$, and when $m_0\equiv 2$
modulo $n$ and $n\leq 4$. Note that a minimal free resolution for
$\fp$ had already been constructed for $n=3$ in \cite{seng} using a
Gr\"{o}bner basis for $\fp$. In \cite{eaca} the following question was
posed~: do the total Betti numbers of $\fp$ depend only on the
integer $m_0$ modulo $n$~? In this article, we answer this question
in affirmative for the above cases. This question will be addressed
in general in our work in progress \cite{hip}.

\section{The defining ideal}

By \cite{patil} and \cite {malooseng}, one knows that the number of
elements in a minimal set of generators of the ideal $\fp$ depends
only on $m_0$ modulo $n$. Our first result shows that $\fp$ has an
additional structure that will be helpful in the sequel.

\begin{theorem} {\label{mingen}}
Let $(\mm) = m_0, \ldots m_n$ be an arithmetic sequence and $\fp$ be
the defining ideal of the monomial curve $C(\mm )$ associated to
$\mm$. Then
 {\footnotesize
$$
\fp= I_2 ({\left(\begin{array}{cccc}
\begin{array}{c}
X_{0}\\[1.5mm]
X_{1}\\[1.5mm]
\end{array} &
\begin{array}{c}
X_{1}\\[1.5mm]
X_{2}\\[1.5mm]
\end{array} &
\begin{array}{c}
\cdots\\[1.5mm]
\cdots\\[1.5mm]
\end{array} &
\begin{array}{c}
X_{n-1}\\[1.5mm]
X_{n}\\[1.5mm]
\end{array}
\end{array}\right)})+ I_2 ({\left(\begin{array}{cccc}
\begin{array}{c}
X_{n}^{a}\\[1.5mm]
X_{0}^{a+d}\\[1.5mm]
\end{array} &
\begin{array}{c}
X_{0}\\[1.5mm]
X_{b}\\[1.5mm]
\end{array} &
\begin{array}{c}
\cdots\\[1.5mm]
\cdots\\[1.5mm]
\end{array} &
\begin{array}{c}
X_{n-b}\\[1.5mm]
X_{n}\\[1.5mm]
\end{array}
\end{array}\right)}).
$$
 }
\end{theorem}

\begin{proof}
It is known from \cite{patil} and \cite {malooseng}, that $\fp$ is
minimally generated by the following set of binomials
$$\{\delta_{ij}\mid 0\leq i < j \leq n-1\}\cup \{\Delta_{1, \,j}\mid
j = 2, \ldots , n+2-b\},$$ such that $$\delta_{i,\,j}:= X_{i}X_{j+1}
- X_{j}X_{i+1}, \quad 0 \leq i < j \leq n-1$$ and $$\Delta_{1,
\,j}:= X_{b+j-2}X_{n}^{a} - X_{j-2}X_{0}^{a+d}, \quad j = 2, \ldots
, n+2-b.$$
The binomials ${\delta_{ij}}$ are precisely the $n\choose
2$ generators of the ideal of $2\times 2$ minors of the matrix
$$A= {\left(\begin{array}{cccc}
\begin{array}{c}
X_{0}\\[1.5mm]
X_{1}\\[1.5mm]
\end{array} &
\begin{array}{c}
X_{1}\\[1.5mm]
X_{2}\\[1.5mm]
\end{array} &
\begin{array}{c}
\cdots\\[1.5mm]
\cdots\\[1.5mm]
\end{array} &
\begin{array}{c}
X_{n-1}\\[1.5mm]
X_{n}\\[1.5mm]
\end{array}
\end{array}\right)}.$$
On the other hand, the binomials
$\Delta _{1,j}$ are the $2\times2$ minors from the first and the $j$th column of the matrix
$$B={\left(\begin{array}{cccc}
\begin{array}{c}
X_{n}^{a}\\[1.5mm]
X_{0}^{a+d}\\[1.5mm]
\end{array} &
\begin{array}{c}
X_{0}\\[1.5mm]
X_{b}\\[1.5mm]
\end{array} &
\begin{array}{c}
\cdots\\[1.5mm]
\cdots\\[1.5mm]
\end{array} &
\begin{array}{c}
X_{n-b}\\[1.5mm]
X_{n}\\[1.5mm]
\end{array}
\end{array}\right)}.$$
Since the rest of the $2\times2$ minors of $B$ are already in the
ideal $I_2(A)$, one gets that $\fp = \fp_{1} + \fp_{2}$, where
$\fp_{1}$ and $\fp_{2}$ are the determinatal ideals generated by the
maximal minors of the matrices $A$ and $B$ respectively.
\end{proof}

\medskip

The actual generators also depend only on the first term $m_0$, the
common difference $d$ and the length $n$ of the arithmetic sequence
$(\mm )$. This is no surprise as these three numbers do determine
the arithmetic sequence. However, we also have that the number of
minimal generators of the ideal $\fp$ is ${n\choose 2} + n-b+1$ and
hence only depends on $n$ and $b$. We conjecture that the following
statement is true:

\begin {conjecture}
Let $(\mm )= m_0,\ldots , m_n$ be an arithmetic sequence. Then all
the Betti numbers of the homogeneous coordinate ring $k[\Gamma]$ of
the affine monomial curve $C(\mm)$ in $\A_k^{n+1}$ associated to
$(\mm )$ are determined by $n$ and the value of $m_0$ modulo $n$.
 \end{conjecture}

In the next section, we discuss the conjecture in cases $b=1,2$ and
$n$ and construct a minimal resolution when $b=1$ or $n$ using the
aforesaid determinantal structures. In the first special case,
$b=1$, we have that $\fp$ is a determinantal ideal. The second
special case is $b=n$, that is when
$$B := {\left(\begin{array}{cccc}
\begin{array}{c}
X_{n}^{a}\\[1.5mm]
X_{0}^{a+d}\\[1.5mm]
\end{array} &
\begin{array}{c}
X_{0}\\[1.5mm]
X_{n}\\[1.5mm]
\end{array}
\end{array}\right)},$$
and therefore the ideal $\fp_{2}$ is the principal ideal generated
by $X_{0}^{a+d+1} - X_{n}^{a+1}$. The third special case is $b=2$,
that is when $k[\Gamma]$ is Gorenstein. In this case, we wiil give the Betti numbers when
the codimension $n$ is 4.

\section{The minimal resolution}

\subsection{First case: determinantal ($m_0\equiv 1$ modulo $n$).}\label{subsec_b=1}

Assume that $b=1$. In this case,
$$
B := {\left(\begin{array}{cccc}
\begin{array}{c}
X_{n}^{a}\\[1.5mm]
X_{0}^{a+d}\\[1.5mm]
\end{array} &
\begin{array}{c}
X_{0}\\[1.5mm]
X_{1}\\[1.5mm]
\end{array} &
\begin{array}{c}
\cdots\\[1.5mm]
\cdots\\[1.5mm]
\end{array} &
\begin{array}{c}
X_{n-1}\\[1.5mm]
X_{n}\\[1.5mm]
\end{array}
\end{array}\right)}
$$
and hence $I_2(A) \subset I_2 (B)$. Thus,
$\fp = \fp_{2}$ is a determinantal ideal of codimension $n$
generated by the maximal minors of the $2\times (n+1)$ matrix $B$.
Therefore, the homogeneous coordinate ring $k[\Gamma]$ is minimally
resolved by the Eagon-Northcott complex.   The resolution is given
by
 {\small
$$
0 \to \wedge^{n+1}R^{n+1}\otimes D_{n-1}(R^{2})\to \cdots \to
\wedge^3R^{n+1}\otimes D_{1}(R^2)\to \wedge^{2}R^{n+1} \to
\wedge^{2}R^{2} \to k[\Gamma] \to 0
$$
 }
such that $D_{s-1}(R^{2})= (S_{s-1}(R^2))^{*}$ denotes the
$R$-module which is the dual of the symmetric algebra of $R^{2}$ and
the module $S_{s-1}(R^{2})$ is free of rank $s$, with a basis the
set of monomials of total degree $(s-1)$ in $\lambda_{0}$,
$\lambda_{1}$ (symbols representing basis elements of $R^{2}$).
Therefore, up to an identification,
$D_{s-1}(R^{2})=(S_{s-1}(R^{2}))^{*}$ is a free $R$-module of rank
$s$, with a basis $\{\lambda_{0}^{v_0}\lambda_{1}^{v_1}\mid v_0, v_1
\in\N \,, v_0+v_1 = s-1\}$. For every $s = 1, \ldots , n$, let
$G_{s}$ denote $\wedge^{s+1}R^{n+1}\otimes D_{s-1}(R^{2})$, which is
a $R$-free module of rank $s{n+1\choose s+1}$ generated by the basis
elements $(e_{i_1}\wedge \cdots \wedge
e_{i_{s+1}})\otimes\lambda_{0}^{v_0}\lambda_{1}^{v_1}$, for every
$1\leq i_1 < i_2 < \cdots < i_{s+1}\leq n+1$ and $v_0, v_1 \in\N \,,
v_0 + v_1 = s-1$. The differentials $d_{s}: G_{s} \rar G_{s-1}$ are
given by the following formulae:
$$d_{1}(e_{i_1}\wedge e_{i_2}) = X_{i_{1}-1}X_{i_{2}} -
X_{i_{2}-1}X_{i_{1}}\,, \quad \forall \ 1\leq i_{1} < i_{2} \leq
n+1,$$ and for $s = 2, \ldots , n$ and $1\leq i_{1} < \cdots <
i_{s+1} \leq n+1$,
 {\footnotesize
\begin{eqnarray*}
d_{s}((e_{i_1}\wedge \cdots \wedge
e_{i_{s+1}})\otimes\lambda_{0}^{v_0}\lambda_{1}^{v_1}) & = &
\sum_{j=1}^{s+1} (-1)^{j+1}X_{i_{j}-1}((e_{i_1}\wedge \cdots
\wedge \widehat{e_{i_{j}}}\wedge\cdots
e_{i_{s+1}})\otimes\lambda_{0}^{v_{0}-1}\lambda_{1}^{v_1})\\ {} &
{} & + \sum_{j=1}^{s+1} (-1)^{j+1}X_{i_{j}}((e_{i_1}\wedge \cdots
\wedge \widehat{e_{i_{j}}}\wedge\cdots
e_{i_{s+1}})\otimes\lambda_{0}^{v_0}\lambda_{1}^{v_{1}-1})\\
\end{eqnarray*}
 }
such that summands on the right hand side involve only
non-negative powers of $\lambda_{0}$ and $\lambda_{1}$.

\medskip

We have therefore proved the following:

\begin{theorem}{\label{b=1}}
The minimal free resolution of the homogeneous coordinate ring
$k[\Gamma]$ of the affine monomial curve $C(\mm)$ in $\A_k^{n+1}$
associated to the arithmetic sequence of integers $(\mm)=m_{0},
\ldots,m_{n}$ with the property $m_{0}\equiv 1$ modulo $n$ is
 {\small
$$
0 \rar R^{n}\stackrel{d_n}{\lar} R^{(n-1){n+1\choose
n}}\stackrel{d_{n-1}}{\lar} \cdots \lar R^{s {n+1\choose s+1}}
\stackrel{d_s}{\lar} \cdots \lar R^{n+1\choose 2}
\stackrel{d_1}{\lar} R \stackrel{\phi}{\lar} k[\Gamma] \rar 0.
$$
 }
In particular, the Betti numbers are $\beta_{0}=1$ and $\beta_{s} =
s{n+1\choose s+1}$ for every $s\in[1,n]$.
\end{theorem}

\begin{remark}\rm
The case $a=1$ (and $b=1$) corresponds to semigroups $\Gamma$ of
maximal embedding dimension, i.e., those semigroups such that the
inequality $m(\Gamma)\geq e(\Gamma)$ is an equality, being
$m(\Gamma):=m_0=an+b$ and $e(\Gamma):=n+1$ the multiplicity and the
embedding dimension respectively.
\end{remark}

\medskip

Note that the above argument gives indeed the graded resolution
(with respect to the grading given by $\deg(X_i)=m_i=m_0+id$):

\begin {corollary} {\label{cor_b=1}}
Under the hypothesis of Theorem~\ref{b=1}, the minimal graded free
resolution of $k[\Gamma]$ is given by

 {\small
\begin{eqnarray*}
0&\lar&\displaystyle{ \bigoplus_{k=1}^{n}R(-(a+n+d)m_0+kd - {n+1
\choose 2}d)
}\\
 &\stackrel{d_{n}}{\lar}&\displaystyle{
\left(\bigoplus_{k=1}^{n-1}R(-nm_0+kd - {n+1 \choose 2} d)\right)
 }\\
 &&\displaystyle{
\oplus \left(\bigoplus_{1\le r_1< \ldots r_{n-1}\le n }\left(
\bigoplus_{k=1}^{n-1} R(-(a+d+n-1)m_0+kd-\sum_{i=1}^{n-1} r_id)
\right)
\right) }\\
 &\stackrel{d_{n-1}}{\lar}&\cdots\\
 &\vdots&\\
\cdots &\stackrel{d_s}{\lar}&\displaystyle{ \left(\bigoplus_{1\le
r_1<\ldots < r_s \le n}\left(\bigoplus_{k=1}^{s-1}R(-sm_0
+kd-\sum_{i=1}^{s} r_i d)\right) \right)\oplus
 }\\
 &&\displaystyle{
\left(\bigoplus_{1\le r_1<\ldots < r_{s-1} \le n}
\left(\bigoplus_{k=1}^{s-1}R(-(a+s-1+d)m_0 +kd-\sum_{i=1}^{s-1} r_i
d)\right)\right)
 }\\
 &\stackrel{d_{s-1}}{\lar}&\cdots\\
 &\vdots&\\
\cdots &\stackrel{d_2}{\lar}&\displaystyle{ \left(\bigoplus_{1\le
r_1<r_2\le n}R(-2m_0-(r_1+r_2-1)d)\right) \oplus
\left(\bigoplus_{k=0}^{n-1} R(-(a+1+d)m_0-kd)\right)
 }\\&
 \stackrel{d_1}{\lar}& \quad R
\quad\stackrel{\varphi}{\lar}\quad k[\Gamma]\rar 0\,.
\end{eqnarray*}
 }
\end {corollary}

\subsection{Second case: almost determinantal ($m_0\equiv n$ modulo $n$).}\label{subsec_b=n}

In this case, we show that $k[\Gamma]$ can be resolved minimally by
a mapping cone  using the resolution for $R/\fp_{1}$, which is the
homogeneous coordinate ring of the rational normal curve. The
minimal free resolution for $R/\fp_{1}$ is the Eagon-Northcott
complex given by
 {\small
$$
\mathcal{E}: 0 \lar \wedge^{n}R^{n}\otimes
(S_{n-2}(R^{2}))^{*}\lar \cdots \lar \wedge^{2}R^{n}\otimes
(S_{0}(R^{2}))^{*} \lar \wedge^{2}R^{2} \lar R/\fp_{1} \lar 0.
$$
 }
Since $E_{s}:= \wedge^{s+1}R^{n}\otimes (S_{s-1}(R^{2}))^{*}=
R^{s{n\choose s+1}}$, for every $s= 1, \ldots , n-1$,
$E_{0}:=\wedge^{2}R^{2}=R$, the resolution $\mathcal{E}$ takes the
form
 {\small
$$
0 \lar R^{n-1}\lar R^{(n-2){n\choose n-1}}\lar
\cdots \lar R^{s {n\choose s+1}} \lar \cdots \lar R^{n\choose 2}
\lar R \lar R/\fp_{1} \lar 0.
$$
 }
The differentials $d_{s}^{1}: E_{s} \rar E_{s-1}$ are defined
similarly as $d_{s}$ above, for every $s= 1, \ldots , n-1$.

\medskip

Now consider the short exact sequence of $R$-modules $$0 \lar
R/(\fp_{1}:\Delta_{1, \,2}) \stackrel{\Delta_{1, \,2}}{\lar}
R/\fp_{1} \lar R/\fp_{1}+ (\Delta_{1, \,2}) \lar 0$$ where the map
$R/(\fp_{1}:\Delta_{1, \,2}) \stackrel{\Delta_{1, \,2}}{\lar}
R/\fp_{1}$ is the multiplication by $\Delta_{1, 2}$. The colon
ideal $(\fp_{1}:\Delta_{1, \,2})$ is exactly equal to $\fp_{1}$,
since $\fp_{1}$ is a prime ideal and $\Delta_{1, \,2}\not\in
\fp_{1}$, being a part of a minimal generating set for the ideal
$\fp = \fp_{1} + \fp_{2}$. Therefore, the above short exact
sequence becomes, $$0 \lar R/\fp_{1}(-(a+d+1)m_{0})
\stackrel{\Delta_{1, \,2}}{\lar} R/\fp_{1} \lar R/\fp_{1}+
(\Delta_{1, \,2}),$$ taking gradation into account. We define the
graded complex homomorphism $$\psi: \mathcal{E}(-(a+d+1)m_{0})
\lar \mathcal{E}$$ as the multiplication by $\Delta_{1, \,2}$,
which is a lift of the map $$R/\fp_{1}(-(a+d+1)m_{0})
\stackrel{\Delta_{1, \,2}}{\lar} R/\fp_{1}.$$ In fact, the complex homomorphism
$\psi $ is simply multiplication by $\Delta _{1,\,2}$.  The mapping cone
$\mathcal{F}(\psi)$ is given by the free modules $F_{s} =
E_{s-1}\oplus E_{s}= R^{(s-1){n\choose s}} \oplus R^{s{n\choose
s+1}}$, for every $s= 0, \ldots , n$ (with $E_{-1} = E_{n}=0$) and
the differentials $d_{s}^{2}: F_{s} \rar F_{s-1}$, defined as
$d_{s}^{2}(x,y) = (\psi_{s-1}(y)+d_{s}^{1}(x),
\,-d_{s-1}^{1}(y))$, for every $s = 1, \ldots , n$.

\medskip

$R/\fp_{1}+(\Delta_{1,\,2})$ is resolved by the mapping cone
$\mathcal{F}(\psi)$.   It is minimal because each of the maps in
$\psi$ is of positive degree and the resolution $\mathcal {E}$ is
minimal.   Hence, the mapping cone $\mathcal{F}(\psi)$ is the
minimal free resolution of the homogeneous coordinate ring
$k[\Gamma] = R/\fp$. We have therefore proved the following:

\begin{theorem}{\label{b=n}}
The minimal free resolution of the homogeneous coordinate ring
$k[\Gamma]$ of the affine monomial curve $C(\mm)$ in $\A_k^{n+1}$
associated to the arithmetic sequence of integers
$(\mm)=m_{0},\ldots,m_{n}$ with the property $m_{0}\equiv n$ modulo
$n$ is
$$
0 \rar R^{n-1}\stackrel{d_n^2}{\lar} R^{(n-2){n\choose n-1}
+ (n-1)}\stackrel{d_{n-1}^2}{\lar} \cdots \lar R^{1 + { n\choose 2}}
\stackrel{d_1^2}{\lar} R \stackrel{\phi}{\lar} k[\Gamma] \rar 0.
$$
In particular, the Betti numbers are $\beta_{0}=1$, $\beta_1=1 + {
n\choose 2}$ and $\beta_{s} = (s-1){n\choose s} + s{n\choose s+1}$
for every $s\in[2,n]$.
\end{theorem}

As in Section~\ref{subsec_b=1}, one can be more precise and give the
graded resolution. First note that the graded minimal resolution of
$\fp_1$ is given by
\begin{eqnarray*}
0&\lar&\displaystyle{
 \bigoplus_{k=1}^{n-1}R(-nm_0+kd - {n+1\choose 2} d)\quad
 \stackrel{d^1_{n-1}}{\lar}\quad\cdots
}\\
&\cdots &\stackrel{d^1_s}{\lar}\displaystyle{ \bigoplus_{1\le
r_1<\ldots < r_s \le n} \left(\bigoplus_{k=1}^{s-1} R(-sm_0 +kd
-\sum r_i d)\right)\quad \stackrel{d^1_{s-1}}{\lar}\quad\cdots
 }\\
&\cdots& \stackrel{d^1_2}{\lar}\displaystyle{ \bigoplus_{1\le
r_1<r_2\le n}R(-2m_0-(r_1+r_2-1)d)
 \quad
 \stackrel{d^1_1}{\lar} \quad R
\quad\lar\quad R/\fp_{1}\rar 0\,.}
\end{eqnarray*}

\begin {corollary} {\label{cor_b=n}}
Under the hypothesis of Theorem~\ref{b=n}, the minimal graded free
resolution of $k[\Gamma]$ is given by
\begin{eqnarray*}
 0&\lar&\displaystyle{
\bigoplus_{k=1}^{n-1}R(-(a+n+d+1)m_0+kd -{n+1\choose 2} d)
 }\\
 &\stackrel{d^2_{n}}{\lar}&\displaystyle{
\left(\bigoplus_{k=1}^{n-1}R(-nm_0+kd - {n+1\choose 2}
d)\right)\oplus
 }\\
 &&\displaystyle{
\left(\bigoplus_{1\le r_1<\ldots < r_{n-1} \le n}\left(\bigoplus
_{k=1}^{n-2}R(-(a+n+d)m_0 +kd -\sum r_i d)\right)\right)
 }\\
 &\stackrel{d^2_{n-1}}{\lar}&\cdots\\
 &\vdots&\\
\cdots &\stackrel{d^2_s}{\lar}&\displaystyle{ \left(\bigoplus_{1\le
r_1<\ldots < r_s \le n}\left(\bigoplus_{k=1}^{s-1} R(-sm_0 +kd -\sum
r_i d)\right) \right) \oplus
 }\\
 &&\displaystyle{
\left(\bigoplus_{1\le r_1<\ldots < r_{s-1} \le n}\left(\bigoplus
_{k=1}^{s-2}R(-(a+s+d)m_0 +kd -\sum r_i d)\right)\right)
 }\\
 &\stackrel{d^2_{s-1}}{\lar}&\cdots\\
 &\vdots&\\
\cdots &\stackrel{d^2_2}{\lar}&\displaystyle{ \left(\bigoplus_ {1\le
r_1<r_2\le n}R(-2m_0-(r_1+r_2-1)d)\right) \oplus R(-(a+1+d)m_0)}\\
 &\stackrel{d^2_1}{\lar}& R \stackrel{\varphi}{\rar} k[\Gamma]\rar 0.
\end{eqnarray*}
\end {corollary}

\subsection{Third case: Gorenstein of codimension 4 ($m_0\equiv 2$ modulo $n$ and $n=4$).}

When $b=2$, the ideal $\fp$ is Gorenstein of height $n$ and hence
the first and the last Betti numbers are known. This provides a very
interesting family of Gorenstein ideals of height $n$ generated by
$(n-1)(n+2)/2$ binomials.

\medskip

When $n=2$, $\fp$ is a complete intersection and when $n=3$, $\fp$
is the ideal of $4\times 4$ pfaffians of a $5\times 5$ skew
symmetric matrix. For $n=4$, the ideal $\fp$ is a height 4
Gorenstein ideal minimally generated by $9$ elements and therefore
has Betti numbers 9, 16, 9 and 1.

\medskip

In fact, we can show that with this weighted grading, the shifts in
a graded resolution of $R/\fp$ are as follows:

\begin{theorem} {\label{gor-res-grad}}
Given two non-negative integers $a$ and $d$, consider the monomial
curve $C(\mm)$ in $\A_k^5$ associated to the arithmetic sequence
$(\mm)=4a+2, 4a+2+d, 4a+2+2d, 4a+2+3d, 4a+2+4d$. Set $q:=2a+1$. Then
the defining ideal $\fp$ of $C(\mm)$ is a height four Gorenstein
ideal whose minimal graded free resolution is

 {\small
\begin{eqnarray*}
0&\rar&\displaystyle{R(-q(q+2d+9)-9d)}\\
 &\rar&\displaystyle{
\left(\bigoplus_{k=7}^9 R(-8q-kd)\right)\oplus \left(\bigoplus
_{k=3}^7R(-q(q+2d+5)-kd)\right) \oplus R(-q(q+2d+5)-5d)
 }\\
 &\rar&\displaystyle{
R(-6q-4d)\oplus\left(\bigoplus_{k=5}^7R^2(-6q-kd)\right)\oplus
R(-6q-8d)\ \oplus R(-q(q+2d+3)-d)
 }\\& &\displaystyle{
\oplus \left(\bigoplus_{k=2}^4R^2(-q(q+2d+3)-kd)\right)\oplus
R(-q(q+2d+3)-5d)
 }\\&\rar&\displaystyle{
\left(\bigoplus_{k=2}^6 R(-4q-kd)\right)\oplus R(-4q-4d)
\oplus\left(\bigoplus_{k=0}^2 R(-q(q+2d+1)-kd)\right)}\ \rar\fp \rar
0.
\end{eqnarray*}
 }
\end{theorem}

\begin {proof}
We know that $\fp$ is Gorenstein of height 4 and by
Theorem~\ref{mingen}, it is generated by the $2\times 2$ minors of
$A$ and $B$, where
$$A = {\left(\begin{array}{cccc}
\begin{array}{c}
X_{0}\\[1.5mm]
X_{1}\\[1.5mm]
\end{array} &
\begin{array}{c}
X_{1}\\[1.5mm]
X_{2}\\[1.5mm]
\end{array} &
\begin{array}{c}
X_{2}\\[1.5mm]
X{3}\\[1.5mm]
\end{array} &
\begin{array}{c}
X_{3}\\[1.5mm]
X_{4}\\[1.5mm]
\end{array}
\end{array}\right)}$$ and
$$B={\left(\begin{array}{cccc}
\begin{array}{c}
X_{4}^a\\[1.5mm]
X_{0}^{a+d}\\[1.5mm]
\end{array} &
\begin{array}{c}
X_{0}\\[1.5mm]
X_{2}\\[1.5mm]
\end{array} &
\begin{array}{c}
X_{1}\\[1.5mm]
X_{3}\\[1.5mm]
\end{array} &
\begin{array}{c}
X_{2}\\[1.5mm]
X_{4}\\[1.5mm]
\end{array}
\end{array}\right)}.$$
Moreover, the $9$ minimal generators of $\fp$ are the six $2\times
2$ minors of $A$ whose degrees are $2m_0+kd$, $k=2,3,4,4,5,6$ for
$m_0:=4a+2=2(2a+1)=2q$, and the $3$ minors of $B$ involving the
first column which are of degrees $m_0(a+d+1)+kd$, $k=0,1,2$. Note
that for both determinantal ideals $I_2(A)$ and $I_2(B)$, the
Eagon-Northcott complex provides a minimal resolution. The $8$
determinantal relations from the Eagon-Northcott resolution of
$I_2(A)$ and the $6$ determinantal relations from the
Eagon-Northcott resolution of $I_2(B)$ involving the first column
will necessarily be among the 16 minimal syzygies of $\fp$. The
degrees of the other two relations are determined by the symmetry of
the resolution. Thus the remaining two relations must be of degrees
$q(q+3)+(2q+2)d$ and $q(q+3)+(2q+4)d$. This determines the rest of
the degrees in the resolution.
\end{proof}

\medskip

We note that when $ d=1$, this ideal can never be determinantal even
though it is height 4 and generated by 9 elements.  That is, this
ideal is not the ideal of $2\times 2$ minors of a $3\times 3$ matrix
for sum of the degrees of such an ideal must be even.    For in this
case, the ideal is generated by 9 elements whose degrees add  up to
$27d+ km_0$.  Since $m_0$ is even, this number is odd whenever $d$
is odd. Hence if $d$ is odd, it cannot be a determinantal ideal.
However, it is not clear if it can be determinantal for even values
of $d$.  Of course, the conjecture only says that the Betti numbers
are determined by $b$.

\bibliographystyle{amsalpha}

\end{document}